\documentclass[11pt, letterpaper]{amsart}
\usepackage[left=1in,right=1in,bottom=1in,top=1in]{geometry}
\usepackage{amsfonts}
\usepackage{amsmath, amssymb}
\usepackage{graphicx}
\usepackage[font=small,labelfont=bf]{caption}
\usepackage{epstopdf}
\usepackage[pdfpagelabels,hyperindex]{hyperref}
\usepackage{xcolor}
\usepackage{amsthm}
\usepackage{float}
\usepackage{pgfplots}
\usepackage{listings}
\usepackage{longtable}
\usepackage{mathrsfs}

\newtheorem{innercustomgeneric}{\customgenericname}
\providecommand{\customgenericname}{}
\newcommand{\newcustomtheorem}[2]{%
  \newenvironment{#1}[1]
  {%
   \renewcommand\customgenericname{#2}%
   \renewcommand\theinnercustomgeneric{##1}%
   \innercustomgeneric
  }
  {\endinnercustomgeneric}
}

\newcustomtheorem{theorem}{Theorem}
\newcustomtheorem{lemma}{Lemma}

\theoremstyle{definition}
\newcustomtheorem{definition}{Definition}
\newcustomtheorem{xca}{Exercise}
\newcustomtheorem{prop}{Proposition}
\newcustomtheorem{corr}{Corollary}

\theoremstyle{remark}
\newcustomtheorem{remark}{Remark}
\numberwithin{equation}{section}

\theoremstyle{plain}
\newcustomtheorem{plain}{Example}
\newtheoremstyle{example}
  {\topsep}
  {\topsep}
  {}
  {}
  {\itshape}
  {}
  {.5em}
  {\thmname{#1}\thmnumber{ #2}\thmnote{ (#3)}}
\theoremstyle{example}

\begin{document}

\title[Non-maximal Closed Prime ideals Are Accessible]{Non-maximal Closed Prime ideals In  A Unital Commutative Banach Algebra Are Accessible}


\author{Ramesh V. Garimella}
\address{Department of Mathematics, Drexel University, Philadelphia, PA, 19104 USA }
\email{rg948@drexel.edu}



\dedicatory{}


\begin{abstract}
It is proved that in a commutative unital Banach algebra, every non-maximal closed prime ideal is accessible. Specifically, it can be represented as the intersection of all closed ideals of the algebra that properly contain it. Consequently, all derivations and epimorphisms on commutative unital semi-prime Banach algebras are continuous. Moreover, any separating ideal in a commutative unital Banach algebra is nilpotent and, therefore, a nil ideal.
\end{abstract}
\maketitle
\section{Introduction}
One key aspect of automatic continuity problems involves identifying conditions on the domain and range of Banach algebras that guarantee the continuity of a homomorphism. One of the earliest automatic continuity results for homomorphisms is attributed to Šilov \cite{silovregular}, who proved that every homomorphism from a Banach algebra $\mathcal{A}$ to a commutative semisimple Banach algebra $\mathcal{B}$ is continuous. Recall that a homomorphism $h$ is a linear operator on $\mathcal{A}$ satisfying $h(xy) = h(x)h(y)$ for all $x, y \in \mathcal{A}$. In \cite{bade1978prime}, Bade and Curtis gave sufficient conditions for the commutative Banach algebras $\mathcal{A}$ and $\mathcal{B}$ to ensure that a homomorphism $h: \mathcal{A} \rightarrow \mathcal{B}$ is continuous. \\  

Recall that a derivation $D$ is a linear map on $\mathcal{A}$ that satisfies $D(ab) = aD(b) + bD(a)$ for all $a, b \in \mathcal{A}$. One of the earliest results on derivations is due to Singer and Wermer \cite{singer1955derivations}, who proved that every continuous derivation on a commutative Banach algebra $\mathcal{A}$ maps into the Jacobson radical of the algebra. They also conjectured that the image of a derivation on any commutative Banach algebra, irrespective of continuity, maps into the Jacobson radical. In 1988, Marc P. Thomas proved \cite{thomas1988image} Singer and Wermer's conjecture. One of the earliest automatic continuity results for derivations was established by B.E. Johnson \cite{johnson1969continuity}, who proved that every derivation on a semisimple Banach algebra is continuous.\\   

For an epimorphism $h$ from $\mathcal{B}$ onto $\mathcal{A}$, the set $\mathcal{S}(h) := \{x \in \mathcal{A} : \exists x_n \rightarrow 0 \text{ in } \mathcal{B}, \text{ with } h(x_n) \rightarrow x \text{ in } \mathcal{A}\}$ is known as the separating ideal of $h$. Similarly, for any derivation $D$ on $\mathcal{A}$, $\mathcal{S}(D) := \{x \in \mathcal{A} : \exists x_n \rightarrow 0 \text{ in } \mathcal{A}, \text{ with } D(x_n) \rightarrow x \text{ in } \mathcal{A}\}$ is the separating ideal of $D$. The separating ideal of an epimorphism or a derivation is a closed ideal of the commutative Banach algebra. By the closed graph theorem, an epimorphism $h: \mathcal{A} \rightarrow \mathcal{B}$ or a derivation $D$ on $\mathcal{A}$ is continuous if and only if the separating ideal is zero. Let $\mathcal{I}$ be a separating ideal of an epimorphism or a derivation as defined above. It was shown in \cite{sinclair1976automatic} that $\mathcal{I}$ satisfies the following stability property: if for each sequence $\{a_n\}$ in $\mathcal{A}$, there exists a natural number $N$ such that if $n > N$, then 
\begin{equation} 
  \label{*} \overline {a_1 \cdots a_n \mathcal{I}} = \overline{a_1 \cdots a_N \mathcal{I}}   
\end{equation}

\noindent where $\overline{J}$ denotes the closure of $J$.\\

Throughout this paper, we assume that all Banach algebras are infinite-dimensional, commutative, and \textit{unital}. Recall that an ideal $P$ in a commutative unital algebra $\mathcal{B}$ is said to be prime if $ab \in P$ implies either $a \in P$ or $b \in P$. It is easy to prove that every maximal ideal of a commutative Banach algebra is prime. A commutative unital Banach algebra is said to be an \textit{integral domain} if the zero ideal is a prime ideal, and the algebra is called \textit{semiprime} if the intersection of all prime ideals is the zero ideal. An element $x$ in a Banach algebra is said to be nilpotent if $x^n = 0$ for some positive integer $n$. An ideal $I$ is said to be \textit{nilpotent} if every element of $I$ is nilpotent. An ideal $I$ is said to be \textit{nil} if $I^n = \{0\}$. Obviously, a nilpotent ideal is contained in every prime ideal of the algebra. \\

Cusack \cite{cusack1977automatic} and Jewell \cite{jewell1977continuity} were the first to establish the connection between non-maximal prime ideals and the automatic continuity of homomorphisms and derivations on commutative Banach algebras. A detailed study of prime ideals in some commutative Banach algebras, especially in the algebra of continuous functions vanishing at infinity on a locally compact space, can be found in Dales' monograph \cite{dales2000banach}. Additionally, we refer to papers by Pham (\cite{pham2008kernels}, \cite{le2010kernels}) on prime ideals and automatic continuity. \\ 

A non-maximal closed prime ideal $P$ is said to be \textit{accessible} if it is the intersection of all closed ideals of the algebra properly containing $P$. Otherwise, the closed prime ideal is said to be \textit{inaccessible}. Accessible prime ideals were first introduced by P.C. Curtis, Jr. \cite{Curtis1981Derivations} to study the image of derivations on commutative Banach algebras with a unique maximal ideal. In fact, he proved that a closed ideal satisfying the stability property \eqref{*} is contained in every accessible prime ideal of the Banach algebra. It was not known whether inaccessible closed prime ideals exist in a commutative Banach algebra.\\

In this paper, using the \textit {Zorn's Lemma}  which is equivalent to \textit{the axiom of choice}, it is proved that all non-maximal closed prime ideals in a unital commutative Banach algebra are accessible. In other words, inaccessible prime ideals do not exist in commutative unital Banach algebras. As a result, affirmative answers are provided for the following seemingly open questions, where questions (1) through (3) are known to be equivalent:
\begin{enumerate}
    \item Is every derivation on a commutative unital Banach algebra that is an integral domain continuous?
    \item Is every derivation on a commutative unital semiprime Banach algebra continuous?
    \item Is the separating ideal of any derivation on a commutative unital Banach algebra nilpotent?
    \item Is every epimorphism from a commutative unital Banach algebra onto an integral domain continuous?
    \item Is every epimorphism from a commutative unital Banach algebra onto a semiprime commutative unital Banach algebra continuous?
    \item Is the separating ideal of any derivation on a commutative unital Banach algebra nilpotent?
    \item Is the separating ideal of an epimorphism from a commutative unital Banach algebra to another commutative unital Banach algebra nilpotent?
\end{enumerate}

\section{Preliminaries}
The following result holds even for commutative rings with identity. Therefore, we state it in a general form.

\begin{prop}{2.1}
\label{prop1}
Let $\mathcal{R}$ be a commutative ring with identity, and let $P$ be a non-maximal prime ideal of $\mathcal{R}$. Then either $P$ is equal to the intersection of all prime ideals of $\mathcal{R}$ properly containing $P$, or there exists a sequence $\{ I_n \}$ of nonzero ideals such that $P = \bigcap_{n = 1}^{\infty} I_n$. 
\end{prop}

\begin{proof} 
We divide the proof into two cases. \\

\textbf{Case I:} Assume $P = \{ 0 \}$. Further suppose the intersection of all nonzero prime ideals of $\mathcal{R}$ is not zero. Let $x$ be a nonzero element in the intersection of all nonzero prime ideals of $\mathcal{R}$. By our assumption, such a nonzero element $x$ exists. Since $\{ 1, x, x^2, \dots \}$ is a multiplicatively closed set not intersecting $\{ 0 \}$, by \textit{Zorn's Lemma}, there exists an ideal $J$ that is maximal with respect to not intersecting $\{ 1, x, x^2, \dots \}$. Moreover, such an ideal $J$ is prime \cite{hungerford2012abstract}. Since $x$ belongs to every nonzero prime ideal of $\mathcal{R}$, we have $J = \{ 0 \}$. Let $I_n = \langle x^n \rangle$, the principal ideal generated by $x^n$ for each $n \geq 1$. To complete the proof, we now show that $\bigcap I_n = \{ 0 \}$. If $\bigcap I_n \neq \{ 0 \}$, then by the maximality of $J$, $\bigcap I_n$ must intersect $\{ x, x^2, x^3, \dots \}$. Thus, there exists $m \geq 1$ such that $x^m \in \bigcap I_n$, which implies $x^m \in \langle x^{m+1} \rangle$. Hence, $x^m = c x^{m+1}$ for some $c \in \mathcal{R}$. Since $\mathcal{R}$ is an integral domain, it follows that $x$ is invertible, which is a contradiction. This completes the proof for this case. \\

\textbf{Case II:} Suppose $P$ is a nonzero prime ideal. Then we apply Case I to the quotient ring $\mathcal{R}/P$ to prove the result.   
\end{proof}

The following technical lemma is useful in the proof of the main result.

\begin{lemma}{2.2}
\label{lemma2}
Let $\mathcal{B}$ be a commutative unital Banach algebra that is an integral domain. Suppose there is a sequence of nonzero ideals intersecting to the zero ideal. Then the zero ideal is an accessible prime ideal of $\mathcal{B}$. 
\end{lemma}

\begin{proof}
The result follows from Remarks 2.2 and 2.3 in \cite{esterle1983elements}.
\end{proof}

The condition stated in the above lemma is a necessary requirement in the case of commutative unital Banach algebras that are separable. A precise formulation of this statement, along with its proof, is provided in the following.

\begin{prop}{2.3}
\label{prop3}
Let $\mathcal{A}$ be a separable commutative Banach algebra. A non-maximal closed prime ideal is accessible \textit{iff} it is the countable intersection of ideals properly containing it.   
\end{prop}

\begin{proof}
If $P$ is an accessible closed prime ideal, then $P = \bigcap I$, where the intersection runs over all the closed ideals $I$ properly containing $P$. Hence,
\begin{equation}
   \mathcal{A} - P = \bigcup_{I \supset P} (\mathcal{A} - I)
\label{b_minus_i}
\end{equation}
Since $\mathcal{A}$ is a separable metric space, it is second countable. Therefore, every open cover has a countable subcover. By \eqref{b_minus_i}, there exists a sequence $\{ I_n \}$ of closed ideals that properly contain $P$ such that $\mathcal{A} - P = \bigcup_{n=1}^{\infty} (\mathcal{A} - I_n)$, or equivalently $P = \bigcap_{n=1}^{\infty} I_n$. 

Conversely, suppose $P$ is the countable intersection of ideals properly containing it. Let $P = \bigcap J_n$, where $J_n$ are ideals properly containing $P$. Note that we are not assuming $J_n$ are closed. Then the quotient space $\mathcal{A}/P$ is a commutative unital Banach algebra such that $\bigcap_{n=1}^{\infty} (J_n/P)$ is the zero ideal in $\mathcal{A}/P$. Since $P$ is properly contained in each $J_n$, $J_n/P$ is a nonzero ideal in $\mathcal{A}/P$ for each $n$. Therefore, by Lemma \eqref{lemma2}, the zero ideal in $\mathcal{A}/P$ is an accessible closed prime ideal. Hence $P$ is an accessible closed prime ideal in $\mathcal{A}$.  
\end{proof}

\begin{prop}{2.4}
\label{prop4}
Let $\mathcal{B}$ be a commutative Banach algebra that is an integral domain. If every prime ideal in $\mathcal{B}$ is closed, then $\{ 0 \}$ is an accessible prime ideal. 
\end{prop}

\begin{proof}
If $\{ 0 \}$ is the intersection of all nonzero prime ideals of $\mathcal{B}$, then $\{ 0 \}$ is obviously an accessible prime ideal by Proposition \eqref{prop1} since every prime ideal is closed.  Otherwise, by Proposition \eqref{prop1}, there exists a sequence $\{ I_n \}$ of nonzero ideals such that $\bigcap I_n = \{ 0 \}$. Then, by Lemma \eqref{lemma2}, $\{ 0 \}$ is an accessible prime ideal. This completes the proof of the proposition. 
\end{proof}

\begin{corr}{2.5}
\label{corr5}
Let $\mathcal{B}$ be a commutative unital Banach algebra in which every prime ideal is closed. Then every non-maximal closed prime ideal is accessible. 
\end{corr}

\begin{proof}
Let $P$ be a non-maximal closed prime ideal of $\mathcal{B}$. Then $\mathcal{B}/P$ is an integral domain that satisfies the hypothesis of Proposition \eqref{prop4}. Hence the zero ideal in $\mathcal{B}/P$ is accessible. Therefore, $P$ is an accessible prime ideal. 
\end{proof}

\section{Main Results}
We are now ready to prove that every non-maximal closed prime ideal in a commutative unital Banach algebra is accessible. In other words, inaccessible closed prime ideals do not exist in any commutative unital Banach algebra.

\begin{theorem}{3.1}
\label{thrm1}
Suppose $\mathcal{A}$ is a commutative unital Banach algebra that is an integral domain. The zero ideal in $\mathcal{A}$ is an accessible prime ideal. 
\end{theorem}

\begin{proof}
Let $I$ be the intersection of all non-zero closed ideals of $\mathcal{A}$. Note that $I$ is the smallest non-zero closed ideal of $\mathcal{A}$. Let $\mathcal{N}$ be the set of all non-closed prime ideals of $\mathcal{A}$. The proof is divided into three cases:\\

\textbf{Case I:} Suppose $\mathcal{N}$ is empty. From Proposition \eqref{prop4}, it follows that the zero ideal is an accessible prime ideal. \\

\textbf{Case II:} Suppose $\mathcal{N}$ is countable, and let $\mathcal{N} = \{P_n\}$, where each $P_n$ is a non-closed prime ideal of $\mathcal{A}$. Define $I_n = I \cap P_n$. Each $I_n$ is non-zero as $\mathcal{A}$ is an integral domain. Note that $\bigcap I_n$ is contained within the intersection of all prime ideals of $\mathcal{A}$. By Proposition \eqref{prop4}, either $\bigcap I_n = \{0\}$ or there exists another sequence of non-zero ideals of $\mathcal{A}$ whose intersection is the zero ideal. In either case, this implies the zero ideal is an accessible prime ideal. \\

\textbf{Case III:} Suppose $\mathcal{N}$ is uncountable. Without loss of generality, we may assume that $\mathcal{N}$ consists exclusively of non-closed prime ideals that do not contain $I$. If $\mathcal{M}$ is the subset of $\mathcal{N}$ consisting of members that do contain $I$, and if $\mathcal{N} \setminus \mathcal{M}$ is countable, then by applying \textbf{Case II} to $\mathcal{N} \setminus \mathcal{M}$, we conclude that the zero ideal is an accessible prime ideal. Let $\mathcal{N} = \{P_{\gamma} : \gamma \in J\}$, where $J$ is an index set.\\

Using natural inclusion as a partial ordering, we can show that the intersection of any chain of prime ideals in $\mathcal{N}$ is itself a member of $\mathcal{N}$, since no element of $\mathcal{N}$ contains a closed ideal. By Zorn's Lemma, $\mathcal{N}$ has minimal elements. Let $P_{\alpha}$ be a fixed minimal element in $\mathcal{N}$.  For each $\beta \neq \alpha$ in $J$, $P_{\alpha} \cap I \neq P_{\beta} \cap I$. If $P_{\alpha} \cap I = P_{\beta} \cap I$, then $P_{\beta} \cap I \subset P_{\alpha}$. Since $P_{\alpha}$ is a prime ideal not containing $I$ and minimal element of $\mathcal{N}$, we conclude $P_{\alpha} = P_{\beta}$. \\

Now choose a non-zero element $x_{\alpha} \in P_{\alpha} \cap I$ and $x_{\beta} \in P_{\beta} \cap I \setminus P_{\alpha} \cap I$. Consider the set $S = \{1\} \cup \{\text{finite products of } x_{\gamma} : \gamma \in J\}$, where $1$ is the multiplicative identity of $\mathcal{A}$.  since $\mathcal{A}$ is an integral domain, $S$ is a multiplicative set not containing zero . Note that  $S$ contains $\{x_{\alpha}^n : n \geq 1\}$ \\

Notice that all non-zero prime ideals of $\mathcal{A}$ (including both closed and non-closed) contain at least one element of $S - \{1\}$. By Zorn's Lemma, the zero ideal is maximal with respect to all ideals not intersecting $S$. \\

Now consider the ideal $\bigcap (x_{\alpha}^n)$, where $(x_{\alpha}^n)$ is the principal ideal generated by $x_{\alpha}^n$. \\

\textit{Claim:} $\bigcap (x_{\alpha}^n) = \{0\}$. \\

\textit{Proof of the claim.} Suppose $\bigcap (x_{\alpha}^n) \neq \{0\}$. Since the zero ideal is maximal among the ideals not intersecting $S$, $S \cap \bigcap (x_{\alpha}^n)$ is non-empty. Let $t \in S$ such that $t \in (x_{\alpha}^n)$ for each positive integer $n$. If $t$ is of the form $x_{\alpha}^m$ for some positive integer $m$, then $x_{\alpha}^m \in (x_{\alpha}^{m+1})$. This means $x_{\alpha}^m = c_{\alpha} x_{\alpha}^{m+1}$ for some $c_{\alpha} \in \mathcal{A}$. This implies $x_{\alpha}$ is invertible, contradicting that $x_{\alpha} \in P_{\alpha}$. \\

If $t$ is of the form $x_{\alpha}^m (x_{\beta_1} \cdots x_{\beta_k})$, where $m$ is a positive integer and $\beta_i \neq \alpha$, then $x_{\alpha}^m (x_{\beta_1} \cdots x_{\beta_k}) \in (x_{\alpha}^{m+1})$. By cancellation in the algebra, $x_{\beta_1} \cdots x_{\beta_k}$ is a multiplier of $x_\alpha$ and hence is in   $P_{\alpha}$. Since $P_{\alpha}$ is prime, some $x_{\beta_i} \in P_{\alpha}$, which is false. This completes the proof of the claim. \\

Since $\bigcap (x_{\alpha}^n) = \{0\}$, by Lemma \eqref{lemma2}, the zero ideal is an accessible prime ideal. This completes the proof of the theorem.
\end{proof}

\begin{theorem}{3.2}
\label{thrm2}
Every non-maximal closed prime ideal in a commutative unital Banach algebra is accessible.
\end{theorem}

\begin{proof}
Let $\mathcal{P}$ be a non-maximal closed prime ideal of a commutative unital Banach algebra $\mathcal{A}$. Then the quotient algebra $\mathcal{A}/\mathcal{P}$ is a unital Banach algebra that is an integral domain. By Theorem \eqref{thrm1}, the zero ideal in $\mathcal{A}/\mathcal{P}$ is accessible. Hence, $\mathcal{P}$ is an accessible prime ideal in $\mathcal{A}$.
\end{proof}

Before we proceed further, we will state a result that was proved by P.C. Curtis in \cite{Curtis1981Derivations} for a separating ideal of a derivation. The proof can easily be modified for any closed ideal that satisfies the stability condition (1.1) in the introduction.  

\begin{prop}{3.3}
\label{prop3b}
Let $\mathcal{I}$ be a closed ideal of a commutative unital Banach algebra $\mathcal{A}$ that satisfies the stability condition (1.1) given in the introduction. Then $\mathcal{I}$ is contained in every accessible prime ideal. 
\end{prop}

\begin{proof}
Refer to the proof of Lemma 1.1 in \cite{Curtis1981Derivations}.
\end{proof}

Now we state and prove a series of results which provide answers to the questions (1) through (7) as stated in the introduction. \\

\begin{theorem}{3.4}
\label{thrm4}
Suppose $\mathcal{A}$ is a commutative unital Banach algebra which is an integral domain. Then every derivation on $\mathcal{A}$ is continuous.  
\end{theorem}

\begin{proof}
Let $D$ be a derivation on the commutative unital Banach algebra $\mathcal{A}$. Let $\mathcal{S}(D)$ be the separating ideal of $D$. As noted earlier, $\mathcal{S}(D)$ is a closed ideal satisfying the stability condition (1.1). Hence, by Proposition \eqref{prop3}, $\mathcal{S}(D)$ is contained in every accessible prime ideal of $\mathcal{A}$. Since the zero ideal is a prime ideal, $\mathcal{S}(D) = \{0\}$. By the closed graph theorem, $D$ is continuous.
\end{proof}

\begin{theorem}{3.5}
\label{thrm4}
Suppose $\mathcal{A}$ is a commutative unital semiprime Banach algebra. Then every derivation on $\mathcal{A}$ is continuous.  
\end{theorem}

\begin{proof}
The proof follows from Proposition (2.4) of \cite{garimella1993nilpotency} or Theorem 1 of \cite{Runde1991AutomaticCO}.
\end{proof}

\begin{theorem}{3.6}
\label{thrm5}
Suppose $\mathcal{A}$ is a commutative unital Banach algebra. Then the separating ideal $\mathcal{S}$ of any derivation $D$ on $\mathcal{A}$ or an epimorphism from a commutative unital Banach algebra onto another commutative unital Banach algebra is a nil ideal. That is, there exists a positive integer $N$ such that $\mathcal{S}^N = \{0\}$. 
\end{theorem}

\begin{proof}
Once again, from Proposition (2.4) of \cite{garimella1993nilpotency} or Theorem 1 of \cite{Runde1991AutomaticCO}, it follows that $\mathcal{S}(D)$ is a nilpotent ideal. Since $\mathcal{S}(D)$ is a closed ideal, the result follows from \cite{grabiner1969nilpotency}.
\end{proof}

\begin{theorem}{3.7}
\label{thrm6}
Every epimorphism from a commutative unital Banach algebra onto another commutative unital Banach algebra, which is an integral domain, is continuous.      
\end{theorem}

\begin{proof}
Since the range is an integral domain, the zero ideal is an accessible prime ideal. Since the separating ideal of the epimorphism satisfies the stability condition (1.1), the result follows from Proposition \eqref{prop3}.
\end{proof}

\section{Final Remarks}

We conclude the paper by presenting examples of unital commutative Banach algebras that satisfy the hypotheses of the main results. \\

Let $\omega$ be a weight over the set of non-negative integers, $\mathbb{Z}^{+}$; that is, $\omega$ is a positive-valued function on $\mathbb{Z}^{+}$ such that $\omega(t_1 + t_2) \leq \omega(t_1)\omega(t_2)$ for all $t_1, t_2 \in \mathbb{Z}^{+}$. We assume $\omega(0) = 1$. A weight $\omega$ over $\mathbb{Z}^{+}$ is said to be \textit{radical} if $\lim_{n \to \infty} (\omega(nt))^{\frac{1}{n}} = 0$ for all $t \in \mathbb{Z}^{+}$.  

For any weight $\omega$ over $\mathbb{Z}^{+}$, let $l^{1}(\mathbb{Z}^{+}, \omega)$ denote the space of all sequences $\{\lambda_n\}_{n \in \mathbb{Z}^{+}}$ of complex numbers such that $\{\lambda_n \omega(n)\}_{n \in \mathbb{Z}^{+}}$ is absolutely summable. Define $\|\{\lambda_n\}\| = \sum_{n=0}^\infty |\lambda_n|\omega(n)$. For $u = \{\lambda_n\}$ and $v = \{\mu_n\} \in l^{1}(\mathbb{Z}^{+}, \omega)$, define their product $u \cdot v = \{\sigma_n\}$, where $\sigma_n = \sum_{r+s=n} \lambda_r \mu_s$. Under this norm and multiplication, $l^{1}(\mathbb{Z}^{+}, \omega)$ forms a unital separable Banach algebra.  

For weights of the form $\omega(n) = e^{-n^c}$ ($c > 1$), $\omega(n) = e^{-n(\log(n+1))^c}$ ($c > 0$), or $\omega(n) = e^{-n\log\log(n+2)}$, the only prime ideals of $l^{1}(\mathbb{Z}^{+}, \omega)$ are the zero ideal and the unique maximal ideal $M_{\omega} := \{\{a_n\}: a_0 = 0\}$. For details, see Theorem 4.6.24 of \cite{dales2000banach}. Notably, this algebra has no non-closed prime ideals, corresponding to Case I of the main results. Moreover, by Theorem \eqref{thrm6}, any epimorphism $h$ from $l^{1}(\mathbb{Z}^{+}, \omega)$ onto a commutative unital Banach algebra $\mathcal{B}$, which is an integral domain, is continuous. Since the kernel of $h$ is a prime ideal of $l^{1}(\mathbb{Z}^{+}, \omega)$, $h$ is either an isomorphism or $\mathcal{B}$ is one-dimensional. \\

L
et $\mathcal{K}$ be an infinite compact space, and let $\mathcal{C(K)}$ denote the commutative unital Banach algebra of continuous complex-valued functions on $\mathcal{K}$ under the supremum norm.  It is well known that  $\mathcal{K}$ contains many non-closed prime ideals. However, this algebra is not an integral domain; it is, nonetheless, semi-simple, meaning the intersection of all maximal ideals is the zero ideal.  Therefore, it would be interesting to construct concrete examples of commutative Banach algebras that are not semi-simple, satisfying Case II and Case III of Theorem \eqref{thrm1}. In particular, this involves finding non-semi-simple commutative unital Banach algebras that are integral domains and have either countably infinite or unaccountably many non-closed prime ideals.\\

\bibliography{paper}

\bibliographystyle{amsplain}

\end{document}